\documentclass[%
 jmp,
 amsmath,amssymb,
preprint,%
]{revtex4-1}

\usepackage[matrix,arrow,curve,cmtip]{xy}
\usepackage{amsthm}
\usepackage{graphicx}
\usepackage{dcolumn}
\usepackage{bm}

\usepackage[utf8]{inputenc}
\usepackage[T1]{fontenc}
\usepackage{mathptmx}
\usepackage{etoolbox}

\newtheorem{theorem}{Theorem}

\newtheorem{corollary}[theorem]{Corollary}
\theoremstyle{definition}
\newtheorem{definition}[theorem]{Definition}
\newtheorem{example}[theorem]{Example}

\makeatletter
\def\@email#1#2{%
 \endgroup
 \patchcmd{\titleblock@produce}
  {\frontmatter@RRAPformat}
  {\frontmatter@RRAPformat{\produce@RRAP{*#1\href{mailto:#2}{#2}}}\frontmatter@RRAPformat}
  {}{}
}%
\makeatother

\begin{document}


\title[A New Approach to Defining Cochain Complexes for Tri-dendriform algebra]{A New Approach to Defining Cochain Complexes for Tri-dendriform algebra}
\author{H. Alhussein$^{1),2)}$}
 \altaffiliation[Corresponding author, ]{email: k.alkhussein@g.nsu.ru}

\affiliation{$^{1)}$ Siberian State University of Telecommunication and Informatics, Novosibirsk, Russia.}%

\affiliation{$^{2)}$Novosibirsk State University of Economics and Management, Novosibirsk, Russia.}

\date{\today}

\begin{abstract}
Our constructions provide a systematic way to study cohomology tri-dendriform algebra via classical cohomology, simplifying computations and enabling the use of established techniques.
\end{abstract}

\maketitle

\section{Introdction}
The study of algebraic structures defined by multiple binary operations has been profoundly advanced by the work of Loday and his school \cite{Loday01}. These so-called Loday-type algebras, which include dialgebras \cite{Loday01}, dendriform algebras \cite{LodayRonco98}, and their generalizations, are not only of intrinsic algebraic interest but also provide a unifying framework for connections with operad theory, cyclic homology, and mathematical physics \cite{LodayVallette12}.

Among these structures, tri-associative and tri-dendriform algebras represent a significant hierarchical level in the "splitting" of associativity. A tri-associative algebra, as defined in \cite{Loday01}, is equipped with three binary operations that satisfy a system of 11 intertwined associativity and distributivity relations. A tri-dendriform (or post-associative) algebra \cite{Loday01} features three operations—\(\prec\) (left), \(\succ\) (right), and \(\cdot\) (middle)—whose sum \(x \circ y = x \prec y + x \succ y + x \cdot y\) is associative, but which individually obey a more complex set of seven identities. These structures have found remarkable applications, notably in the combinatorics of planar trees \cite{LodayRonco98} and within the Hopf algebraic approach to renormalization in quantum field theory \cite{ConnesKreimer00}. In this physical context, the three operations of a tri-dendriform algebra provide the fundamental combinatorial rules for assembling and disentangling the nested divergences of Feynman diagrams during renormalization.

A central challenge in the exploration of any such algebraic structure is the development of its cohomology theory. Cohomology groups are indispensable for classifying deformations \cite{GerstenhaberSchack92}, understanding extensions, and identifying obstructions. For classical associative algebras, the Hochschild cohomology provides a complete and well-understood framework \cite{Weibel94}. However, for tri-algebraic structures, the cochain complexes are derived directly from the defining multilinear relations, leading to differentials that are sums over all possible placements of a new variable within the multi-operational context \cite{MarklShnider95}. This inherent complexity often renders direct computation prohibitive, creating a significant barrier to a deeper homological understanding of these algebras.

In this paper, we introduce a novel methodology that circumvents this complexity by embedding the problem into the well-mapped territory of Hochschild cohomology. Our approach is built upon two principal contributions. First, we establish a general tensor product construction: we prove that the tensor product of a commutative triassociative algebra and a tridendriform algebra canonically carries the structure of an associative algebra. This result is non-trivial, relying on a precise interplay between the compatibility laws of the two constituent algebras.

Second, and most crucially, we leverage this construction to define an explicit cochain map
\[
\Psi : C^*_{\mathrm{tri}}(B,N) \hookrightarrow C^*_{\mathrm{Hoch}}(A \otimes B, A \otimes N),
\]
where \(A\) is a free commutative tri-associative algebra and \(N\) is a module over \(B\). The map \(\Psi\) is defined combinatorially, summing over all partitions of the inputs and applying the tri-associative operations in \(A\) to the tri-dendriform cochain in \(B\). A detailed and intricate computation confirms that \(\Psi\) commutes with the differentials, establishing it as an injective morphism of cochain complexes.

The profound implication is that the cohomology of the more complex tridendriform structure can be studied through the lens of classical Hochschild cohomology. Our method thus provides a systematic pathway to harness the powerful machinery already developed for associative algebras—including long exact sequences, spectral sequences, and computational tools \cite{Weibel94}—for the analysis of operadic cohomologies. This not only simplifies computations but also enables new ones previously out of reach.

To illustrate the efficacy of our approach, we provide detailed examples, computing the second cohomology group of specific finite-dimensional tri-dendriform algebras by analyzing their image under \(\Psi\) in the Hochschild complex. In summary, this work bridges a significant gap in the homological theory of Loday-type algebras, offering a unifying strategy for reducing complex operadic cohomologies to more manageable classical ones \cite{LodayVallette12}.

\section{Commutative tri-algebra and tri-dendriform  algebras}
\begin{definition} [\cite{ LodayVallette12}]
 A commutative tri-algebra is a vector space  \( A \) equipped with two binary operations \( \ast \) and \( \bullet \) such that:

\begin{itemize}
    \item \( (A, \ast) \) is a Perm-algebra:
    \[
    (x \ast y) \ast z = x \ast (y \ast z) = x \ast (z \ast y),
    \]
    
    \item \( (A, \bullet) \) is a commutative algebra:
    \[
    x \bullet y = y \bullet x,
    \]
    \[
    (x \bullet y) \bullet z = x \bullet (y \bullet z),
    \]
    
    \item The two operations \( \ast \) and \( \bullet \) must verify the following compatibility relations:
    \[
    x \ast (y \bullet z) = x \ast (y \ast z),
    \]
    \[
    (x \bullet y) \ast z = x \bullet (y \ast z).
    \]
\end{itemize}
\end{definition}

\begin{definition}[\cite{ LodayVallette12}]
 A vector space \( B \) over a field \( \Bbbk \) with three bilinear operations \( \succ, \prec, \cdot \) is called a tri-dendriform (post-associative) algebra, if the following identities are satisfied for all \( x, y, z \in B \):

\begin{align*}
(x \prec y) \prec z & = x \prec (y \prec z+ y \succ z+y \cdot z), \\ 
(x \succ y) \prec z & = x \succ (y \prec z), \\ 
(x \prec y+ x \succ y+x \cdot y) \succ z & = x \succ (y \succ z), \\ 
x \succ (y \cdot z) & = (x \succ y) \cdot z, \\ 
(x \prec y) \cdot z & = x \cdot (y \succ z), \\ 
(x \cdot y) \prec z & = x \cdot (y \prec z), \\ 
(x \cdot y) \cdot z & = x \cdot (y \cdot z).
\end{align*}
In particular, the \textbf{total product} $x \circ y := x \prec y + x \succ y+x.y$ is associative.
\end{definition}

\begin{example}\label{exp:tridend}
Let    $B$ be a $1$-dimensional tri-dendriform algebra with basis $\{e\}$ and operations:
\begin{align*}
e \prec e = e \quad 
, e \succ e = e, \quad  e.e=-e.
\end{align*}
\end{example}

\begin{example}\label{exp:tridend1}
Let    $B$ be a $2$-dimensional tri-dendriform algebra with basis $\{e_1,e_2\}$ and operations:

\[
e_1 \prec e_1 = e_1,\quad e_1 \succ e_1 = e_1,\quad e_1 \cdot e_1 = -e_1
\]
\[
e_2 \prec e_2 = e_2,\quad e_2 \succ e_2 = e_2,\quad e_2 \cdot e_2 = -e_2
\]
\[
e_1 \prec e_2 = 0,\quad e_1 \succ e_2 = 0,\quad e_1 \cdot e_2 = 0
\]
\[
e_2 \prec e_1 = 0,\quad e_2 \succ e_1 = 0,\quad e_2 \cdot e_1 = 0.
\]

\end{example}

\section{Associative and tri-dendriform  algebra}
\begin{theorem}
Let $(A, \ast, \bullet)$ be a commutative tri-algebra and $( B,\succ, \prec, \cdot )$ be a  tridendriform algebra. Define the product on $A \otimes B$ by:
\[
(a_1 \otimes b_1). (a_2 \otimes b_2) = (a_1 * a_2) \otimes (b_1 \prec b_2) + (a_2 * a_1) \otimes (b_1 \succ b_2) + (a_1 \bullet a_2) \otimes (b_1 \cdot b_2)
\]
Then $(A \otimes B, .)$ is an associative algebra.
\end{theorem}

\begin{proof}
We prove that the product on \(A \otimes B\) is associative by verifying:
\[
\big((a_1 \otimes b_1) \cdot (a_2 \otimes b_2)\big) \cdot (a_3 \otimes b_3) = (a_1 \otimes b_1) \cdot \big((a_2 \otimes b_2) \cdot (a_3 \otimes b_3)\big)
\]
For \(x = a_1 \otimes b_1\) and \(y = a_2 \otimes b_2\), the product is:
\[
x \cdot y = (a_1 \ast a_2) \otimes (b_1 \prec b_2) + (a_2 \ast a_1) \otimes (b_1 \succ b_2) + (a_1 \bullet a_2) \otimes (b_1 \cdot b_2)
\]
\textbf{ Expansion of \((x \cdot y) \cdot z\):}
First compute \(x \cdot y\) as above, then multiply by \(z = a_3 \otimes b_3\):
\[
\begin{aligned}
(x \cdot y) \cdot z = &\big[(a_1 \ast a_2) \otimes (b_1 \prec b_2)\big] \cdot (a_3 \otimes b_3) \\
+ &\big[(a_2 \ast a_1) \otimes (b_1 \succ b_2)\big] \cdot (a_3 \otimes b_3) \\
+ &\big[(a_1 \bullet a_2) \otimes (b_1 \cdot b_2)\big] \cdot (a_3 \otimes b_3)
\end{aligned}
\]

Expanding each term using the product definition:

\begin{align*}
= ((a_1 \ast a_2) \ast a_3) \otimes \big((b_1 \prec b_2) \prec b_3\big) \quad (L1)\\
+ (a_3 \ast (a_1 \ast a_2)) \otimes \big((b_1 \prec b_2) \succ b_3\big) \quad (L2)\\
+ ((a_1 \ast a_2) \bullet a_3) \otimes \big((b_1 \prec b_2) \cdot b_3\big) \quad (L3)\\
+ ((a_2 \ast a_1) \ast a_3) \otimes \big((b_1 \succ b_2) \prec b_3\big) \quad (L4)\\
+ (a_3 \ast (a_2 \ast a_1)) \otimes \big((b_1 \succ b_2) \succ b_3\big) \quad (L5)\\
+ ((a_2 \ast a_1) \bullet a_3) \otimes \big((b_1 \succ b_2) \cdot b_3\big) \quad (L6)\\
+((a_1 \bullet a_2) \ast a_3) \otimes \big((b_1 \cdot b_2) \prec b_3\big) \quad (L7)\\
+ (a_3 \ast (a_1 \bullet a_2)) \otimes \big((b_1 \cdot b_2) \succ b_3\big) \quad (L8)\\
+ ((a_1 \bullet a_2) \bullet a_3) \otimes \big((b_1 \cdot b_2) \cdot b_3\big) \quad (L9)
\end{align*}

\textbf{ Expansion of \(x \cdot (y \cdot z)\):}
First compute \(y \cdot z\):
\[
y \cdot z = (a_2 \ast a_3) \otimes (b_2 \prec b_3) + (a_3 \ast a_2) \otimes (b_2 \succ b_3) + (a_2 \bullet a_3) \otimes (b_2 \cdot b_3)
\]

Then multiply by \(x = a_1 \otimes b_1\):
\[
\begin{aligned}
x \cdot (y \cdot z) = &(a_1 \otimes b_1) \cdot \big[(a_2 \ast a_3) \otimes (b_2 \prec b_3)\big] \\
+ &(a_1 \otimes b_1) \cdot \big[(a_3 \ast a_2) \otimes (b_2 \succ b_3)\big] \\
+ &(a_1 \otimes b_1) \cdot \big[(a_2 \bullet a_3) \otimes (b_2 \cdot b_3)\big]
\end{aligned}
\]

Expanding each term:

\begin{align*}
= (a_1 \ast (a_2 \ast a_3)) \otimes \big(b_1 \prec (b_2 \prec b_3)\big) \quad (R1)\\
+ ((a_2 \ast a_3) \ast a_1) \otimes \big(b_1 \succ (b_2 \prec b_3)\big) \quad (R2)\\
+ (a_1 \bullet (a_2 \ast a_3)) \otimes \big(b_1 \cdot (b_2 \prec b_3)\big) \quad (R3)\\
+ (a_1 \ast (a_3 \ast a_2)) \otimes \big(b_1 \prec (b_2 \succ b_3)\big) \quad (R4)\\
+ ((a_3 \ast a_2) \ast a_1) \otimes \big(b_1 \succ (b_2 \succ b_3)\big) \quad (R5)\\
+ (a_1 \bullet (a_3 \ast a_2)) \otimes \big(b_1 \cdot (b_2 \succ b_3)\big) \quad (R6)\\
+ (a_1 \ast (a_2 \bullet a_3)) \otimes \big(b_1 \prec (b_2 \cdot b_3)\big) \quad (R7)\\
+ ((a_2 \bullet a_3) \ast a_1) \otimes \big(b_1 \succ (b_2 \cdot b_3)\big) \quad (R8)\\
+ (a_1 \bullet (a_2 \bullet a_3)) \otimes \big(b_1 \cdot (b_2 \cdot b_3)\big) \quad (R9)
\end{align*}

We'll now match corresponding terms using the axioms:
\begin{itemize}
    \item {Matching \((L1)\) and \((R1 + R4+R7)\)}
    
-By tri-dendriform axiom:
$
(b_1 \prec b_2) \prec b_3 = b_1 \prec (b_2 \prec b_3 + b_2 \succ b_3+b_2 \cdot b_3)
$

-By commutative tri-algebra axiom:
$
(a_1 \ast a_2) \ast a_3 = a_1 \ast (a_2 \ast a_3) = a_1 \ast (a_3 \ast a_2)=a_1 \ast (a_2 \bullet a_3)
$

\item {Matching \((L2+L5+L8)\) and \((R5)\)}

-By tri-dendriform axiom: $
  (b_1 \prec b_2 + b_1 \succ b_2+b_1 \cdot b_2)\succ b_3=b_1 \succ (b_2 \succ b_3)
$

-By commutative tri-algebra axiom:
$
(a_3 \ast a_2) \ast a_1 = a_3 \ast (a_2 \ast a_1) = a_3 \ast (a_1 \ast a_2)=a_3 \ast (a_1 \bullet a_2)
$

\item {Matching \((L3)\) and \((R6)\)}

-By tri-dendriform axiom: $
  (b_1 \prec b_2)\cdot b_3 =b_1 \cdot (b_2 \succ b_3) 
$

-By commutative tri-algebra axiom:
$
(a_1\ast a_2)\bullet a_3=a_1\bullet (a_3\ast a_2)
$

\item {Matching \((L4)\) and \((R2)\)}

-By tri-dendriform axiom: $
  (b_1 \succ b_2)\prec b_3 =b_1 \succ (b_2 \prec b_3) 
$

-By commutative tri-algebra axiom:
$
(a_2\ast a_1)\ast a_3=(a_2\ast a_3) \ast a_1
$

\item {Matching \((L6)\) and \((R8)\)}

-By tri-dendriform axiom: $
  (b_1 \succ b_2)\cdot b_3 =b_1 \succ (b_2 \cdot b_3) 
$

-By commutative tri-algebra axiom:
$
(a_2\ast a_1)\ast a_3=(a_2\ast a_3) \ast a_1=a_2*(a_3 \bullet a_1)
$

\item {Matching \((L7)\) and \((R3)\)}

-By tri-dendriform axiom: $
  b_1 \cdot (b_2\prec b_3) =(b_1 \cdot b_2) \prec b_3 
$

-By commutative tri-algebra axiom:
$
(a_1\bullet a_2)\ast a_3=a_1\bullet (a_2 \ast a_3)
$

\item {Matching \((L9)\) and \((R9)\)}

-By tri-dendriform axiom: $
  b_1 \cdot (b_2\cdot b_3) =(b_1 \cdot b_2) \cdot b_3 
$

-By commutative tri-algebra axiom:
$
(a_1\bullet a_2)\bullet a_3=a_1\bullet (a_2 \bullet a_3)
$
\end{itemize}

\end{proof}

\begin{theorem}\label{5.4}
Let $(A, \ast, \bullet)$ be a free commutative tri- algebra and $(B, \succ, \prec, \cdot)$ be a tri-dendriform algebra. The product on $A \otimes B$:
\[
(a_1 \otimes b_1) (a_2 \otimes b_2) = (a_1 \ast a_2) \otimes (b_1 \prec b_2) + (a_2 \ast a_1) \otimes (b_1 \succ b_2) + (a_1 \bullet a_2) \otimes (b_1 \cdot b_2)
\]
 is an associative. Let $N$ is a $B$-module. Then there is injective cochain map: 
\[
\Psi:
C^*_{\mathrm{tri}}(B,N) \hookrightarrow C^*_{\mathrm{Hoch}}(A\otimes B,A\otimes N),
\]

which sends a tri-dendriform cocycle class  to Hochschild cocycle class.
\end{theorem}
\begin{proof}
     Given a tri-dendriform  cochain $g \in C^n_{\mathrm{dend}}(B, N)$, for $a_1,\ldots,a_n\in A$ and $b_1,\ldots,b_n\in B$, define:
\[
(\Psi g)(a_1 \otimes b_1, \dots, a_n \otimes b_n) = \sum_{k=1}^n \sum_{1 \le i_1 < \dots < i_k \le n} P(a_1, \dots, a_n; i_1, \dots, i_k) \otimes g((i_1, \dots, i_k), b_1, \dots, b_n),
\]

where the element \(P(a_1, \dots, a_n; i_1, \dots, i_k) \in A\) is constructed as follows:

-  \(I = \{i_1, \dots, i_k\}\)  with \(i_1 < \dots < i_{k}\).

- \(J = \{1, \dots, n\} \setminus I = \{j_1, \dots, j_{n-k}\}\).

-
  \[
  P(a_1, \dots, a_n; I) = (a_{i_1} \bullet \cdots \bullet a_{i_k}) \ast (a_{j_1} \ast \cdots \ast a_{j_{n-k}}).
  \]

Let us to show that $\Psi$ commutes with the  differentials:
\[
\delta_{HH}(\Psi f)=\Psi(\delta_{dend}f)
\]
For $n=2$, we have:

\[
(\Psi g)(a_1 \otimes b_1) = P(a_1; 1) \otimes g(1, b_1).
\]
Here \(I = \{1\}\), \(J = \emptyset\), so \(P(a_1; 1) = a_1\) (since \(\bullet\)-product of one element is itself, and \(\ast\)-product of empty set is omitted).

So:
\[
(\Psi g)(a_1 \otimes b_1) = a_1 \otimes g(1, b_1).
\]

Apply \(\delta_{\mathrm{HH}}\):

\[
(\delta_{\mathrm{HH}} (\Psi g))(x_1, x_2) = \underbrace{x_1 \cdot (\Psi g)(x_2)}_{T_0} - \underbrace{(\Psi g)(x_1 x_2)}_{T_1} + \underbrace{(\Psi g)(x_1) \cdot x_2}_{T_2},
\]
where \(x_i = a_i \otimes b_i\).

\textbf{1.Term $(T_0)$: $x_1  (\Psi g)(x_2)$}
\begin{align*}
&=(a_1 \otimes b_1)(\Psi g)(a_2\otimes b_2)\\
&=(a_1 \otimes b_1)(a_2\otimes g(1, b_2)\\
&= (a_1 \ast a_2) \otimes (b_1 \prec g(1, b_2)) \;+\; (a_2 \ast a_1) \otimes (b_1 \succ g(1, b_2)) \;+\; (a_1 \bullet a_2) \otimes (b_1 \cdot g(1, b_2)).
\end{align*}

\textbf{2. Term $(T_1)$: \(-(\Psi g)(x_1 x_2)\)}
\begin{align*}
    &=-(\Psi g)\big((a_1\otimes b_1)(a_2 \otimes b_2)\big)\\
    &=-(\Psi g)\big((a_1\ast a_2)\otimes (b_1 \prec b_2)+(a_2\ast a_1)\otimes (b_1 \succ b_2)+(a_1\bullet a_2)\otimes (b_1. b_2)\big)\\
    &=-(a_1\ast a_2)\otimes g(1,b_1\prec b_2)-(a_2\otimes a_)\otimes g(1,b_1\succ b_2)-(a_1\o a_2)\otimes g(1,b_1\prec b_2)
\end{align*}

\textbf{3.Term $(T_2)$: $(\Psi g)(x_1)x_2$}
\begin{align*}
& = g(a_1\otimes b_1)(a_2\otimes b_2)\\
&=(a_1\otimes g(1,b_1))(a_2\otimes b_2)\\
&=(a_1\ast a_2)\otimes (g(1,b_1)\prec b_2)+(a_2\ast a_1)\otimes (g(1,b_1)\succ b_2)+(a_1\bullet a_2)\otimes (g(1,b_1). b_2)
\end{align*}

So:
\[
\begin{aligned}
\delta_{\mathrm{HH}}(\Psi g)(x_1, x_2) = &\ (a_1 \ast a_2) \otimes \big[ b_1 \prec g(1, b_2) - g(1, b_1 \prec b_2) + g(1, b_1) \prec b_2 \big] \\
+ &\ (a_2 \ast a_1) \otimes \big[ b_1 \succ g(1, b_2) - g(1, b_1 \succ b_2) + g(1, b_1) \succ b_2 \big] \\
+ &\ (a_1 \bullet a_2) \otimes \big[ b_1 \cdot g(1, b_2) - g(1, b_1 \cdot b_2) + g(1, b_1) \cdot b_2 \big]\\
&= \sum_{I \subset \{1,2\}, |I|\ge 1} P(a_1, a_2; I) \otimes (\delta_{\mathrm{tri}} g)(I; b_1, b_2)..
\end{aligned}
\]

- \(I = \{1\}\): \(P = a_1 \ast a_2\), \(\delta_{\mathrm{tri}} g\) as above  matches first line.

- \(I = \{2\}\): \(P = a_2 \ast a_1\), matches second line.

- \(I = (1,2)\): \(P = a_1 \bullet a_2\), matches third line.
Thus:
\[
\delta_{\mathrm{HH}}(\Psi g) = \Psi(\delta_{\mathrm{tri}} g) \quad \text{for } n=2.
\]
For $n=3$, we have:

\[
(\Psi g)(a_1 \otimes b_1, a_2 \otimes b_2)
= \sum_{k=1}^2 \sum_{1 \le i_1 < \dots < i_k \le 2} P(a_1, a_2; i_1, \dots, i_k) \otimes g((i_1, \dots, i_k), b_1, b_2).
\]

Possible \(I\):

- \(|I|=1\): \(I = \{1\}, \{2\}\)

- \(|I|=2\): \(I = \{1,2\}\)

Compute \(P\) for each:

- \(I = \{1\}\): \(J = \{2\}\),  
\(P = a_1 \ast a_2 \)

- \(I = \{2\}\): \(J = \{1\}\),  
\(P = a_2 \ast a_1 \)  

- \(I = \{1,2\}\): \(J = \emptyset \),  
\(P = a_1 \bullet a_2\)

So:
\[
\begin{aligned}
(\Psi g)(a_1\otimes b_1, a_2\otimes b_2) = &\ (a_1 \ast a_2 ) \otimes g(1; b_1, b_2) \\
+ &(a_2 \ast a_1) \otimes g(2; b_1, b_2) \\
+ & (a_1 \bullet a_2) \otimes g((1,2); b_1, b_2).
\end{aligned}
\]

Apply \(\delta_{\mathrm{HH}}\):

\[
(\delta_{\mathrm{HH}} (\Psi g))(x_1, x_2,x_3) = \underbrace{x_1 \cdot (\Psi g)(x_2,x_3)}_{T_0} - \underbrace{(\Psi g)(x_1 x_2,x_3)}_{T_1} +\underbrace{(\Psi g)(x_1, x_2x_3)}_{T_2}- \underbrace{(\Psi g)(x_1,x_2) \cdot x_3}_{T_3},
\]
where \(x_i = a_i \otimes b_i\).

\textbf{1.Term $(T_0)$: $x_1  (\Psi g)(x_2,x_3)$}
\[
\Psi(g)(x_2,x_3) = (a_2 \ast a_3) \otimes g(1,b_2,b_3) + (a_3 \ast a_2) \otimes g(2,b_2,b_3) + (a_2 \bullet a_3) \otimes g((1,2),b_2,b_3)
\]

Multiply by \(x = a_1 \otimes b_1\):

\[
\begin{aligned}
x_1  \Psi(g)(x_2,x_3) = &\ (a_1 \ast (a_2 \ast a_3)) \otimes (b_1 \prec g(1,b_2,b_3)) \\
+ &\ ((a_2 \ast a_3) \ast a_1) \otimes (b_1 \succ g(1,b_2,b_3)) \\
+ &\ (a_1 \bullet (a_2 \ast a_3)) \otimes (b_1 \cdot g(1,b_2,b_3)) \\
+ &\ (a_1 \ast (a_3 \ast a_2)) \otimes (b_1 \prec g(2,b_2,b_3)) \\
+ &\ ((a_3 \ast a_2) \ast a_1) \otimes (b_1 \succ g(2,b_2,b_3)) \\
+ &\ (a_1 \bullet (a_3 \ast a_2)) \otimes (b_1 \cdot g(2,b_2,b_3)) \\
+ &\ (a_1 \ast (a_2 \bullet a_3)) \otimes (b_1 \prec g((1,2),b_2,b_3)) \\
+ &\ ((a_2 \bullet a_3) \ast a_1) \otimes (b_1 \succ g((1,2),b_2,b_3)) \\
+ &\ (a_1 \bullet (a_2 \bullet a_3)) \otimes (b_1 \cdot g((1,2),b_2,b_3))
\end{aligned}
\]

\textbf{2. Term $(T_1)$: \(-\Psi(g)(x_1x_2, x_3)\)}

Compute \(x_1x_2\):

\[
x_1x_2= (a_1 \ast a_2) \otimes (b_1 \prec b_2) + (a_2 \ast a_1) \otimes (b_1 \succ b_2) + (a_1 \bullet a_2) \otimes (b_1 \cdot b_2)
\]

Apply \(\Psi(g)\):

\[
\begin{aligned}
-\Psi(g)(x_1x_2, x_3) = &\ - (a_1 \ast a_2 \ast a_3) \otimes g(1, b_1 \prec b_2, b_3) \\
&\ - (a_3 \ast (a_1 \ast a_2)) \otimes g(2, b_1 \prec b_2, b_3) \\
&\ - ((a_1 \ast a_2) \bullet a_3) \otimes g((1,2), b_1 \prec b_2, b_3) \\
&\ - (a_2 \ast a_1 \ast a_3) \otimes g(1, b_1 \succ b_2, b_3) \\
&\ - (a_3 \ast (a_2 \ast a_1)) \otimes g(2, b_1 \succ b_2, b_3) \\
&\ - ((a_2 \ast a_1) \bullet a_3) \otimes g((1,2), b_1 \succ b_2, b_3) \\
&\ - ((a_1 \bullet a_2) \ast a_3) \otimes g(1, b_1 \cdot b_2, b_3) \\
&\ - (a_3 \ast (a_1 \bullet a_2)) \otimes g(2, b_1 \cdot b_2, b_3) \\
&\ - ((a_1 \bullet a_2) \bullet a_3) \otimes g((1,2), b_1 \cdot b_2, b_3)
\end{aligned}
\]

\textbf{3. Term $(T_2)$: \(+\Psi(g)(x_1, x_2x_3)\)}

Compute \(x_2x_3\):

\[
x_2x_3 = (a_2 \ast a_3) \otimes (b_2 \prec b_3) + (a_3 \ast a_2) \otimes (b_2 \succ b_3) + (a_2 \bullet a_3) \otimes (b_2 \cdot b_3)
\]

Apply \(\Psi(g)\):

\[
\begin{aligned}
\Psi(g)(x_1, x_2x_3) = &\ (a_1 \ast (a_2 \ast a_3)) \otimes g(1,b_1, b_2 \prec b_3) \\
+ &\ ((a_2 \ast a_3) \ast a_1) \otimes g(2,b_1, b_2 \prec b_3) \\
+ &\ (a_1 \bullet (a_2 \ast a_3)) \otimes g((1,2),b_1, b_2 \prec b_3) \\
+ &\ (a_1 \ast (a_3 \ast a_2)) \otimes g(1,b_1, b_2 \succ b_3) \\
+ &\ ((a_3 \ast a_2) \ast a_1) \otimes g(2,b_1, b_2 \succ b_3) \\
+ &\ (a_1 \bullet (a_3 \ast a_2)) \otimes g((1,2),b_1, b_2 \succ b_3) \\
+ &\ (a_1 \ast (a_2 \bullet a_3)) \otimes g(1,b_1, b_2 \cdot b_3) \\
+ &\ ((a_2 \bullet a_3) \ast a_1) \otimes g(2,b_1, b_2 \cdot b_3) \\
+ &\ (a_1 \bullet (a_2 \bullet a_3)) \otimes g((1,2),b_1, b_2 \cdot b_3)
\end{aligned}
\]

\textbf{4. Fourth Term: \(-\Psi(g)(x_1,x_2) x_3\)}

Compute \(\Psi(g)(x_1,x_2)\):

\[
\Psi(g)(x_1,x_2) = (a_1 \ast a_2) \otimes g(1,b_1,b_2) + (a_2 \ast a_1) \otimes g(2,b_1,b_2) + (a_1 \bullet a_2) \otimes g((1,2),b_1,b_2)
\]

Multiply by \(x_3 = a_3 \otimes b_3\):

\[
\begin{aligned}
- \Psi(g)(x_1,x_2)  x_3 = &\ - (a_1 \ast a_2 \ast a_3) \otimes (g(1,b_1,b_2) \prec b_3) \\
&\ - (a_3 \ast (a_1 \ast a_2)) \otimes (g(1,b_1,b_2) \succ b_3) \\
&\ - ((a_1 \ast a_2) \bullet a_3) \otimes (g(1,b_1,b_2) \cdot b_3) \\
&\ - (a_2 \ast a_1 \ast a_3) \otimes (g(2,b_1,b_2) \prec b_3) \\
&\ - (a_3 \ast (a_2 \ast a_1)) \otimes (g(2,b_1,b_2) \succ b_3) \\
&\ - ((a_2 \ast a_1) \bullet a_3) \otimes (g(2,b_1,b_2) \cdot b_3) \\
&\ - ((a_1 \bullet a_2) \ast a_3) \otimes (g((1,2),b_1,b_2) \prec b_3) \\
&\ - (a_3 \ast (a_1 \bullet a_2)) \otimes (g((1,2),b_1,b_2) \succ b_3) \\
&\ - ((a_1 \bullet a_2) \bullet a_3) \otimes (g((1,2),b_1,b_2) \cdot b_3)
\end{aligned}
\]
Let us compute each term in \(\delta_{HH}(\Psi g)\) one by one.
\[
\begin{aligned}
\delta_{HH}(\Psi(g))\big(x_1 \otimes b_1, x_2 \otimes b_2,x_{3}\otimes b_{3}\big)&=\ (a_1 \ast (a_2 \ast a_3)) \otimes (b_1 \prec g(1,b_2,b_3)) \\
+ &\ (a_1 \ast (a_3 \ast a_2)) \otimes (b_1 \prec g(2,b_2,b_3)) \\
+ &\ (a_1 \ast (a_2 \bullet a_3)) \otimes (b_1 \prec g(3,b_2,b_3)) \\
 & - (a_1 \ast a_2 \ast a_3) \otimes g(1, b_1 \prec b_2, b_3) \\
&+\ (a_1 \ast (a_2 \ast a_3)) \otimes g(1,b_1, b_2 \prec b_3) \\
+ &\ (a_1 \ast (a_3 \ast a_2)) \otimes g(1,b_1, b_2 \succ b_3) \\
+ &\ (a_1 \ast (a_2 \bullet a_3)) \otimes g(1,b_1, b_2 \cdot b_3) \\
 &\ - ((a_1 \ast a_2) \ast a_3) \otimes (g(1,b_1,b_2) \prec b_3) \\
 &\ +((a_2 \ast a_3) \ast a_1) \otimes (b_1 \succ g(1,b_2,b_3)) \\
&\ - (a_2 \ast a_1 \ast a_3) \otimes g(1, b_1 \succ b_2, b_3) \\
+ &\ ((a_2 \ast a_3) \ast a_1) \otimes g(2,b_1, b_2 \prec b_3) \\
&\ - (a_2 \ast a_1 \ast a_3) \otimes (g(2,b_1,b_2) \prec b_3) \\
+ &\ ((a_3 \ast a_2) \ast a_1) \otimes (b_1 \succ g(2,b_2,b_3)) \\
&\ - (a_3 \ast (a_1 \ast a_2)) \otimes g(2, b_1 \prec b_2, b_3) \\
\end{aligned}
\]
\[
\begin{aligned}
&\ - (a_3 \ast (a_2 \ast a_1)) \otimes g(2, b_1 \succ b_2, b_3) \\
&\ - (a_3 \ast (a_1 \bullet a_2)) \otimes g(2, b_1 \cdot b_2, b_3) \\
+ &\ ((a_3 \ast a_2) \ast a_1) \otimes g(2,b_1, b_2 \succ b_3) \\
&\ - (a_3 \ast (a_1 \ast a_2)) \otimes (g(1,b_1,b_2) \succ b_3) \\
&\ - (a_3 \ast (a_2 \ast a_1)) \otimes (g(2,b_1,b_2) \succ b_3) \\
&\ - (a_3 \ast (a_1 \bullet a_2)) \otimes (g(3,b_1,b_2) \succ b_3) \\
+ &\ (a_1 \bullet (a_2 \ast a_3)) \otimes (b_1 \cdot g(1,b_2,b_3)) \\
&-((a_1 \bullet a_2) \ast a_3) \otimes g(1, b_1 \cdot b_2, b_3) \\
&+ \ (a_1 \bullet (a_2 \ast a_3)) \otimes g(3,b_1, b_2 \prec b_3) \\
&\ - ((a_1 \bullet a_2) \ast a_3) \otimes (g(3,b_1,b_2) \prec b_3) \\
&+ \ (a_1 \bullet (a_3 \ast a_2)) \otimes (b_1 \cdot g(2,b_2,b_3)) \\
&\ - ((a_1 \ast a_2) \bullet a_3) \otimes g(3, b_1 \prec b_2, b_3) \\
&+ \ (a_1 \bullet (a_3 \ast a_2)) \otimes g(3,b_1, b_2 \succ b_3) \\
&\ - ((a_1 \ast a_2) \bullet a_3) \otimes (g(1,b_1,b_2) \cdot b_3) \\
+ &\ ((a_2 \bullet a_3) \ast a_1) \otimes (b_1 \succ g(3,b_2,b_3)) \\
& - ((a_2 \ast a_1) \bullet a_3) \otimes g(3, b_1 \succ b_2, b_3) \\
+ & ((a_2 \bullet a_3) \ast a_1) \otimes g(2,b_1, b_2 \cdot b_3) \\
& - ((a_2 \ast a_1) \bullet a_3) \otimes (g(2,b_1,b_2) \cdot b_3) \\
+ &\ (a_1 \bullet (a_2 \bullet a_3)) \otimes (b_1 \cdot g(3,b_2,b_3))\\
&- ((a_1 \bullet a_2) \bullet a_3) \otimes g(3, b_1 \cdot b_2, b_3)\\
+ &\ (a_1 \bullet (a_2 \bullet a_3)) \otimes g(3,b_1, b_2 \cdot b_3)\\
- &\ (a_1 \bullet a_2) \bullet a_3 \otimes g(3,b_1, b_2 ) b_3)
\end{aligned}
\]
By using the tri-algebra identities:
\[
\begin{aligned}
    &-a_1*(a_2*a_3)=a_1*(a_3*a_2)=(a_1*a_2)*a_3=a_1*(a_2\bullet a_3)\\
    &-a_2*(a_1*a_3)=a_2*(a_3*a_1)=(a_2*a_1)*a_3\\
    \end{aligned}
\]
\[
\begin{aligned}
    &-a_3*(a_2*a_1)=a_3*(a_1*a_2)=(a_3*a_2)*a_1=a_3*(a_1\bullet a_2)\\
      &-(a_1\bullet a_2)\ast a_3=a_1\bullet (a_2\ast a_3)\\
    &-a_1\bullet (a_3\ast a_2)=(a_1\ast a_2)\bullet a_3\\
    &-(a_2\bullet a_3) \ast a_1=(a_2 \ast a_1)\bullet a_3\\
    &-(a_1\bullet a_2)\bullet a_3=a_1\bullet (a_2\bullet a_3)\\
\end{aligned}
\]
we get:
\[
\begin{aligned}
    \delta_{HH}(\Psi(g))\big(x_1 \otimes b_1, x_2 \otimes b_2,x_{3}\otimes b_{3}\big)&=\ (a_1 \ast (a_2 \ast a_3)) \otimes \Big(b_1 \prec g(1,b_2,b_3) \\
& +b_1 \prec g(2,b_2,b_3) +b_1 \prec g(3,b_2,b_3) \\
 & -  g(1, b_1 \prec b_2, b_3) + g(1,b_1, b_2 \prec b_3) \\
 &+ g(1,b_1, b_2 \succ b_3) +  g(1,b_1, b_2 \cdot b_3) \\
 &- g(1,b_1,b_2) \prec b_3\Big) \\
  &\ +((a_2 \ast a_1) \ast a_3) \otimes \Big(b_1 \succ g(1,b_2,b_3) \\
& - g(1, b_1 \succ b_2, b_3) + g(2,b_1, b_2 \prec b_3) \\
&-  g(2,b_1,b_2) \prec b_3\Big) \\
+ &\ ((a_3 \ast a_2) \ast a_1) \otimes \Big(b_1 \succ g(2,b_2,b_3) \\
&- g(2, b_1 \prec b_2, b_3)-  g(2, b_1 \succ b_2, b_3) \\
&-  g(2, b_1 \cdot b_2, b_3) + g(2,b_1, b_2 \succ b_3) \\
& -g(1,b_1,b_2) \succ b_3  - g(2,b_1,b_2) \succ b_3  \\
&- g(3,b_1,b_2) \succ b_3\Big) \\
 &+((a_1 \bullet a_2) \ast a_3)) \otimes \Big(b_1 \cdot g(1,b_2,b_3)\\
&- g(1, b_1 \cdot b_2, b_3) 
+ g(3,b_1, b_2 \prec b_3) \\
& - g(3,b_1,b_2) \prec b_3\Big) \\
+ & ((a_1 \bullet a_3) \ast a_2) \otimes \Big(b_1 \cdot g(2,b_2,b_3) \\
&- g(3, b_1 \prec b_2, b_3) + g(3,b_1, b_2 \succ b_3) \\
& -  g(1,b_1,b_2) \cdot b_3\Big) \\
\end{aligned}
\]
\[
\begin{aligned}
&+ ((a_2 \bullet a_3) \ast a_1) \otimes \Big(b_1 \succ g(3,b_2,b_3) \\
&\ - g(3, b_1 \succ b_2, b_3) +g(2,b_1, b_2 \cdot b_3) \\
& -  (g(2,b_1,b_2) \cdot b_3\Big) \\
+ &\ (a_1 \bullet (a_2 \bullet a_3)) \otimes \Big(b_1 \cdot g(3,b_2,b_3)\\
&-  g(3, b_1 \cdot b_2, b_3)+ \otimes g(3,b_1, b_2 \cdot b_3)\\
&-g(3,b_1, b_2) \cdot b_3\Big)\\
&  = (a_1 \ast a_2 \ast a_3) \otimes \Psi(\delta_{tri}g)(1,b_1, b_2) \\
&+ (a_2 \ast a_1 \ast a_3) \otimes \Psi(\delta_{tri}g)(2,b_1, b_2)\\
&+ (a_3 \ast a_2 \ast a_1) \otimes \Psi(\delta_{tri}g)(3,b_1, b_2,b_3)\\
&+ ((a_1 \bullet a_2) \ast a_3) \otimes \Psi(\delta_{tri}g)(4,b_1, b_2,b_3)\\
&+ ((a_1 \bullet a_3) \ast a_2) \otimes \Psi(\delta_{tri}g)(5,b_1, b_2,b_3)\\
&+ ((a_2 \bullet a_3) \ast a_1) \otimes \Psi(\delta_{tri}g)(6,b_1, b_2,b_3)\\
&+(a_1 \bullet a_2 \bullet a_3) \otimes \Psi(\delta_{tri}g)(7,b_1, b_2,b_3)\\
&=\Psi (\delta_{tri}g)(a_1 \otimes b_1,a_2 \otimes b_2,a_3\otimes b_3)
\end{aligned}
\]
Therefore:
\[
\delta_{\mathrm{HH}}(\Psi g) = \Psi(\delta_{\mathrm{tri}} g)
\]

In general case: For \(g \in C_{\mathrm{tridend}}^n(B, B)\), we define the Hochschild cochain \(\Psi g\) by its action on elementary tensors \(x_j = a_j \otimes b_j\):

\[
(\Psi g)(a_1 \otimes b_1, \dots, a_n \otimes b_n) = \sum_{k=1}^n \sum_{1 \le i_1 < \dots < i_k \le n} P(a_1, \dots, a_n; i_1, \dots, i_k) \otimes g((i_1, \dots, i_k), b_1, \dots, b_n),
\]

where the element \(P(a_1, \dots, a_n; i_1, \dots, i_k) \in A\) is constructed as follows:

-  \(I = \{i_1, \dots, i_k\}\)  with \(i_1 < \dots < i_{k}\).

- \(J = \{1, \dots, n\} \setminus I = \{j_1, \dots, j_{n-k}\}\).

-
  \[
  P(a_1, \dots, a_n; I) = (a_{i_1} \bullet \cdots \bullet a_{i_k}) \ast (a_{j_1} \ast \cdots \ast a_{j_{n-k}}).
  \]

 The Hochschild Differential \(\delta_{\mathrm{HH}}(\Psi g)\) is:
\[
\begin{aligned}
(\delta_{\mathrm{HH}} (\Psi g))(x_1, \dots, x_{n+1}) = &\ x_1 \cdot (\Psi g)(x_2, \dots, x_{n+1}) \\
&+ \sum_{i=1}^n (-1)^i (\Psi g)(x_1, \dots, x_i x_{i+1}, \dots, x_{n+1}) \\
&+ (-1)^{n+1} (\Psi g)(x_1, \dots, x_n) \cdot x_{n+1},
\end{aligned}
\]
where \(x_j = a_j \otimes b_j\).

 We write:
\[
\delta_{\mathrm{HH}}(\Psi g)(x_1, \dots, x_{n+1}) = T_0 + \sum_{i=1}^n T_i + T_{n+1}.
\]

\textbf{ Term \(T_0\):}

\[
T_0 = x_1 \cdot (\Psi g)(x_2, \dots, x_{n+1}).
\]

First, compute the inner part:
\[
(\Psi g)(a_2 \otimes b_2, \dots, a_{n+1} \otimes b_{n+1}) = \sum_{k=1}^{n} \sum_{2 \le i_1 < \dots < i_k \le n+1} P(a_2, \dots, a_{n+1}; i_1, \dots, i_k) \otimes g((i_1, \dots, i_k), b_2, \dots, b_{n+1}).
\]
Here, the indices \(i_1, \dots, i_k\) refer to positions in the \(n\)-tuple \((a_2, \dots, a_{n+1})\).

Let \(P' = P(a_2, \dots, a_{n+1}; I')\)  \(I' = \{i_1, \dots, i_k\} \subset \{2, \dots, n+1\}\). Using the definition of the product in $A\otimes B$, we get three types of terms :

- Term (1a): \((a_1 \ast P') \otimes (b_1 \prec g(I'; b_2, \dots, b_{n+1}))\)

- Term (1b): \((P' \ast a_1) \otimes (b_1 \succ g(I'; b_2, \dots, b_{n+1}))\)

- Term (1c): \((a_1 \bullet P') \otimes (b_1 \cdot g(I'; b_2, \dots, b_{n+1}))\)

Thus,
\[
\begin{aligned}
T_0 = \sum_{k=1}^{n} \sum_{1 \le i_1 < \dots < i_k \le n} \bigg[ &(a_1 \ast P') \otimes (b_1 \prec g(I'; b_2, \dots, b_{n+1})) \\
+ &(P' \ast a_1) \otimes (b_1 \succ g(I'; b_2, \dots, b_{n+1})) \\
+ &(a_1 \bullet P') \otimes (b_1 \cdot g(I'; b_2, \dots, b_{n+1})) \bigg].
\end{aligned}
\]

 \textbf{The Middle Terms \(T_i\) for \(1 \le i \le n\):}

\[
T_i = (-1)^i (\Psi g)(x_1, \dots, x_i x_{i+1}, \dots, x_{n+1}).
\]

The product \(x_i x_{i+1}\) expands as:

- Case 1: \((a_i \ast a_{i+1}) \otimes (b_i \prec b_{i+1})\)

- Case 2: \((a_{i+1} \ast a_i) \otimes (b_i \succ b_{i+1})\)

- Case 3: \((a_i \bullet a_{i+1}) \otimes (b_i \cdot b_{i+1})\)

Let \(y_1, \dots, y_n\) be the list where \(y_j = x_j\) for \(j < i\), \(y_i\) is one of the three cases above, and \(y_j = x_{j+1}\) for \(j > i\).

Then \(T_i = (-1)^i [T_i^{(1)} + T_i^{(2)} + T_i^{(3)}]\), corresponding to the three cases.

For a fixed case \(m \in \{1,2,3\}\) and a subset \(I \subset \{1, \dots, n\}\) (positions in the \(y\)-list), we have:
\[
T_i^{(m)} = \sum_{I \subset \{1, \dots, n\}} P^{(m)}(a_1, \dots, a_{n+1}; I, i) \otimes g(I; b_1, \dots, b_{i-1}, b_i \circ_m b_{i+1}, b_{i+2}, \dots, b_{n+1}),
\]
where:
- \(\circ_1 = \prec\), \(\circ_2 = \succ\), \(\circ_3 = \cdot\)
- \(P^{(m)}\) is the \(A\)-part, which depends on whether \(i \in I\):

  - If \(i \notin I\), then \(y_i\) (which involves \(a_i \ast a_{i+1}\) or \(a_{i+1} \ast a_i\) or \(a_i \bullet a_{i+1}\)) is part of the \(\ast\)-product of the remaining elements.
  
  - If \(i \in I\), then \(y_i\) is one of the factors in the \(\bullet\)-product.

Thus, explicitly:
\[
\begin{aligned}
T_i = (-1)^i \sum_{I \subset \{1, \dots, n\}} \bigg[ &P^{(1)}(a_1, \dots, a_{n+1}; I, i) \otimes g(I; b_1, \dots, b_{i-1}, b_i \prec b_{i+1}, b_{i+2}, \dots, b_{n+1}) \\
+ &P^{(2)}(a_1, \dots, a_{n+1}; I, i) \otimes g(I; b_1, \dots, b_{i-1}, b_i \succ b_{i+1}, b_{i+2}, \dots, b_{n+1}) \\
+ &P^{(3)}(a_1, \dots, a_{n+1}; I, i) \otimes g(I; b_1, \dots, b_{i-1}, b_i \cdot b_{i+1}, b_{i+2}, \dots, b_{n+1}) \bigg].
\end{aligned}
\]

\textbf{ The Last Term \(T_{n+1}\):}

\[
T_{n+1} = (-1)^{n+1} (\Psi g)(x_1, \dots, x_n) \cdot x_{n+1}.
\]

We have:
\[
(\Psi g)(a_1 \otimes b_1, \dots, a_n \otimes b_n) = \sum_{k=1}^{n} \sum_{1 \le i_1 < \dots < i_k \le n} P(a_1, \dots, a_n; i_1, \dots, i_k) \otimes g((i_1, \dots, i_k), b_1, \dots, b_n).
\]

Let \(P'' = P(a_1, \dots, a_n; I)\) for \(I \subset \{1, \dots, n\}\). Multiplying on the right by \(a_{n+1} \otimes b_{n+1}\) gives three types of terms:

- Term ((n+1)a): \((P'' \ast a_{n+1}) \otimes (g(I; b_1, \dots, b_n) \prec b_{n+1})\)

- Term ((n+1)b): \((a_{n+1} \ast P'') \otimes (g(I; b_1, \dots, b_n) \succ b_{n+1})\)

- Term ((n+1)c): \((P'' \bullet a_{n+1}) \otimes (g(I; b_1, \dots, b_n) \cdot b_{n+1})\)

Thus,
\[
\begin{aligned}
T_{n+1} = (-1)^{n+1} \sum_{k=1}^{n} \sum_{1 \le i_1 < \dots < i_k \le n} \bigg[ &(P'' \ast a_{n+1}) \otimes (g(I; b_1, \dots, b_n) \prec b_{n+1}) \\
+ &(a_{n+1} \ast P'') \otimes (g(I; b_1, \dots, b_n) \succ b_{n+1}) \\
+ &(P'' \bullet a_{n+1}) \otimes (g(I; b_1, \dots, b_n) \cdot b_{n+1}) \bigg].
\end{aligned}
\]

The total expression for \(\delta_{\mathrm{HH}}(\Psi g)\) is a sum over many terms of the form $(Element\  of A) \otimes (Element\  of B)$. 
For a fixed \(I = \{i_1, \dots, i_k\} \subset \{1, \dots, n+1\}\), the coefficient of \(P(a_1, \dots, a_{n+1}; I)\) in the \(A\)-part must be exactly \((\delta_{\mathrm{tri}} g)(I; b_1, \dots, b_{n+1})\) in the \(B\)-part by using tri-algebra indentias.

Let's illustrate that:

For \(I = \{1\}\):

From the term \(T_0\), each element of form:
\[
a_1 \ast P(a_2, \cdots, a_{n+1}, i_1, \cdots, i_k)
\]
(for all \(k = 1, \dots, n\) and  \(1 \le i_1 < \dots < i_k \le n\)) equals \(P(a_1, \cdots, a_{n+1}, 1)\).

Therefore, in \((\delta g)(1, b_1, b_2, \cdots, b_{n+1})\), there is a corresponding term for each such combination:
\[
b_1 \prec g((i_1, \cdots, i_k), b_2, \cdots, b_{n+1})
\]
for each \(k = 1, \dots, n\) and each \(1 \le i_1 < \dots < i_k \le n\).

This is why the complete contribution from \(T_0\) appears as the sum over all these possibilities:
\[
b_1 \prec \left( \sum_{k=1}^{n} \sum_{2 \le i_1 < \dots < i_k \le n} g((i_1, \cdots, i_k), b_2, \cdots, b_{n+1}) \right)
\]
in  \((\delta_{tri} g)(1, b_1, b_2, \cdots, b_{n+1})\).

 From \(T_i\):

The contributions from the terms \(T_i\) are separated into two distinct cases based on the index.

Case 1: (\(i = 1\))

The relation is given by:
\[
P(a_1 \ast a_2, a_3, \cdots, a_{n+1}, 1) = P(a_1, a_2, a_3, \cdots, a_{n+1}, 1)
\]
This identity contributes the following single term to the coboundary \((\delta g)(1, b_1, b_2, \cdots, b_{n+1})\):
\[
-g(1, b_1 \prec b_2, b_3, \cdots, b_{n+1})
\]

Case 2: (\(i = 2, \dots, n-1\))

For these internal indices, the contributions come from three separate sums. Each individual term within the following expressions:
\[
\begin{aligned}
&(a) \ -\sum_{i=2}^{n-1} P(a_1,\cdots, a_i\ast a_{i+1},\cdots, a_{n+1},1)\\
&(b) \ -\sum_{i=2}^{n-1} P(a_1,\cdots, a_{i+1}\ast a_{i},\cdots, a_{n+1},1)\\
&(c) \ -\sum_{i=2}^{n-1} P(a_1,\cdots, a_i\bullet a_{i+1},\cdots, a_{n+1},1)
\end{aligned}
\]
is equal to \(P(a_1, a_2, a_3, \cdots, a_{n+1}, 1)\).

Therefore, in the coboundary \((\delta g)(1, b_1, \cdots, b_{n+1})\), this corresponds to the following three sums, where each term in the \(a\)-variable sum has a direct counterpart in the \(b\)-variable expression:
\[
\begin{aligned}
&(-1)^i\sum_{i=2}^{n-1} g(1, b_1,\cdots, b_i \prec b_{i+1}, \cdots, b_{n+1})\\
&(-1)^i\sum_{i=2}^{n-1} g(1, b_1,\cdots, b_i \succ b_{i+1}, \cdots, b_{n+1})\\
&(-1)^i\sum_{i=2}^{n-1} g(1, b_1,\cdots, b_i \cdot b_{i+1}, \cdots, b_{n+1})
\end{aligned}
\]

From \(T_{n+1}\):

The final term \(T_{n+1}\) is defined by the relation:
\[
P(a_1, \cdots, a_n, 1) \ast a_{n+1} = P(a_1, \cdots, a_{n+1}, 1)
\]
This contributes the following term to the coboundary \((\delta g)(1, b_1, \cdots, b_{n+1})\):
\[
g(1, b_1, \cdots , b_n) \prec b_{n+1}
\]

In this case the tri-dendriform differential
may be expressed via:
\[
\begin{aligned}\label{eq:Dend-diff-details}
  (\delta_{tri}g)(1,b_1, \ldots,&  b_{n+1}) =  \sum_{i=1}^na_1\prec f(i,b_2,\ldots,b_{n+1})- f(1,b_1\prec b_2,\ldots,b_{n+1})\\
    &+ \sum_{i=2}^n(-1)^if(1,b_1,\ldots, b_i\circ b_{i+1},\ldots,b_{n+1})+(-1)^{n+1}  f(1,b_1,b_2,\ldots,b_{n})\prec b_{n+1}\\
     \end{aligned}
   \]

For $i_1=2$. We consider the operation in the A-part:

\[
P(a_1, a_2, \dots, a_{n+1}, 2) =a_2 \ast a_1 \ast a_3 \cdots a_{n+1}
\]
Since
\[
\text{A-part operation: } a_2 \ast a_1 \ast a_3 \cdots a_{n+1} \quad \longleftrightarrow \quad \text{B-part coboundary: } (\delta g)(2, b_1, \cdots, b_{n+1})
\]

This, in turn, corresponds precisely to the following expression in the B-part coboundary:

\[
\begin{aligned}
(\delta_{tri} g)(2, b_1, \cdots, b_{n+1}) =&\quad b_1 \succ g(1, b_2, \ldots, b_{n+1}) \\
&\quad - g(1, b_1 \succ b_2, b_3, \ldots, b_{n+1}) \\
&\quad + g(2, b_1, b_{2} \prec b_{3}, \ldots, b_{n+1}) \\
&\quad + \sum_{i=3}^{n} (-1)^i g(2, b_1, \ldots, b_i \circ b_{i+1}, \ldots, b_{n+1}) \\
&\quad + (-1)^{n+1} \left( g(2, b_1, \ldots, b_n) \prec b_{n+1} \right)
\end{aligned}
\]
Where $b_i\circ b_{i+1}=b_i\prec b_{i+1}+b_i\succ b_{i+1}+b_i.b_{i+1}$

\[
\text{A-part operation: } a_r \ast a_1 \ast a_2 \cdots a_{n+1} \quad \longleftrightarrow \quad \text{B-part coboundary: } (\delta g)(r, b_1, \cdots, b_{n+1})
\]
This, in turn, corresponds precisely to the following expression in the B-part coboundary:
\[
   \begin{aligned}
  (\delta_{tri}g)(r,b_1, \ldots,&  b_{n+1}) =   \sum_{i=1}^{r-1}b_1\succ g(r-1,b_2,\ldots,b_{n+1})\\
   &+\sum_{i=1}^{r-2}(-1)^if(r-1,b_1,\ldots, b_i\circ b_{i+1},\ldots,b_{n+1})\\
    &+(-1)^{r-1} g(r-1,b_1,\ldots,b_{r-1}\succ b_r,\ldots,b_{n+1})\\
     &+(-1)^{r} g(r,b_1,\ldots,b_{r}\prec b_{r+1},\ldots,b_{n+1})\\
     &+ \sum_{i=r+1}^{n}(-1)^if(r,b_1,\ldots, b_i\circ b_{i+1},\ldots,b_{n+1})\\
     &+\sum_{i=r+1}^{n}(-1)^{n+1}g(r,b_1,\ldots,b_{n})\prec b_{n+1}\\
     \end{aligned}
\]
For $r=3,\cdots,n$.

\[
\text{A-part operation: } a_{n+1} \ast a_n \ast  \cdots \ast a_{1} \quad \longleftrightarrow \quad \text{B-part coboundary: } (\delta g)(n+1, b_1, \cdots, b_{n+1})
\]
This, in turn, corresponds precisely to the following expression in the B-part coboundary:
\[
  \begin{aligned}
     (\delta_{tri}g)(n+1,b_1, \ldots,&  b_{n+1}) = b_1\succ g(n,b_2,\ldots,b_{n+1})\\
     &+\sum_{i=1}^{n-1}(-1)^i g(n,b_1,\ldots, b_i\circ b_{i+1},\ldots,b_{n+1})\\
     &+(-1)^{n} g(n,b_1,b_2,\ldots,b_n\succ b_{n+1})\\
&+\sum_{i=1}^n(-1)^{n+1}g(i,b_1,b_2,\ldots,b_n)\succ b_{n+1}.
\end{aligned}
\]
\[
\text{A-part operation: } (a_{1} \bullet a_2) \ast a_3  \cdots \ast a_{n+1} \quad \longleftrightarrow \quad \text{B-part coboundary: } (\delta_{tri} g)((1,2), b_1, \cdots, b_{n+1})
\]
This, in turn, corresponds precisely to the following expression in the B-part coboundary:

\[
\begin{aligned}
    (\delta_{tri} g)((1,2),b_1,\cdots,b_{n+1})&=b_1.g(1,b_2,\cdots,b_{n+1})-g(1,b_1.b_2,b_3,\cdots,b_{n+1})\\
    &+g((1,2),b_1,b_2\prec b_3,\cdots,b_{n+1})+\sum_{i=3}^{n}(-1)^ig((1,2),b_1,\cdots,b_i\circ b_{i+1},\cdots,b_{n+1})\\
    &+(-1)^{n+1}g((1,2),b_1,\cdots,b_n)\prec b_{n+1}
\end{aligned}
\]  

 \[
 \text{A-part operation: } (a_{1} \bullet a_2\bullet a_3) \ast a_4  \cdots \ast a_{n+1} \quad \longleftrightarrow \quad \text{B-part coboundary: } (\delta_{tri} g)((1,2,3), b_1, \cdots, b_{n+1})
\]
This, in turn, corresponds precisely to the following expression in the B-part coboundary:

\[
\begin{aligned}
    (\delta_{tri} g)((1,2,3),b_1,\cdots,b_{n+1})&=b_1.g((1,2),b_2,\cdots,b_{n+1})-g((1,2),b_1.b_2,b_3,\cdots,b_{n+1})\\
    &+g((1,2),b_1,b_2.b_3,\cdots,b_{n+1})-g((1,2),b_1,b_2,b_3 \prec b_4, b_5,\cdots,b_{n+1})\\
    &+\sum_{i=4}^{n}(-1)^ig((1,2,3),b_1,\cdots,b_i\circ b_{i+1},\cdots,b_{n+1})\\
    &+(-1)^{n+1}g((1,2,3),b_1,\cdots,b_n)\prec b_{n+1}
\end{aligned}
\]

We have the general correspondence:

\[
\text{A-part: } (a_{1} \bullet a_{2} \bullet \cdots \bullet a_{k}) \ast a_{k+1} \ast \cdots \ast a_{{n+1}} 
\quad \longleftrightarrow \quad 
\text{B-part: } (\delta_{\text{tri}} g)((1, \dots, k), b_1, \cdots, b_{n+1})
\]
\[
\begin{aligned}
    (\delta_{tri} g)((i_1,i_2,\ldots,i_k),b_1,\cdots,b_{n+1})&=b_1.g((i_1,i_2,\ldots,i_{k-1}),b_2,\cdots,b_{n+1})\\
&+\sum_{i=1}^{k-1}(-1)^ig((i_1,i_2,\ldots,i_{k-1}),b_1,\cdots,b_i.b_{i+1},\cdots,b_{n+1})\\
  &+(-1)^{k}g((i_1,i_2,\ldots,i_{k-1}),b_1,\cdots,b_{k}\prec b_{k+1},\cdots,b_{n+1})\\
    &+\sum_{i=k+1}^{n}(-1)^ig((i_1,i_2,\ldots,i_k),b_1,\cdots,b_i\circ b_{i+1},\cdots,b_{n+1})\\
    &+(-1)^{n+1}g((i_1,i_2,\ldots,i_k),b_1,\cdots,b_n)\prec b_{n+1}
\end{aligned}
\]  
For $k=1,\cdots,n$. The finial $k=n+1$:

\[
\text{A-part: } a_{1} \bullet a_{2} \bullet \cdots \bullet a_{n+1}  
\quad \longleftrightarrow \quad 
\text{B-part: } (\delta_{\text{tri}} g)((1, \dots, n+1), b_1, \cdots, b_{n+1})
\]

\[
\begin{aligned}
    (\delta_{tri} g)((1,2,\cdots,n+1),b_1,\cdots,b_{n+1})&=b_1.g((1,2,\cdots,n),b_2,\cdots,b_{n+1})\\
    &+\sum_{i=1}^{n}(-1)^ig((1,2,\cdots,n),b_1,\cdots,b_i. b_{i+1},\cdots,b_{n+1})\\
    &+(-1)^{n+1}g((1,2,\cdots,n),b_1,\cdots,b_n). b_{n+1}
\end{aligned}
\]

In this we way, we find that the coefficient of \(P(a_1, \dots, a_{n+1}; i_1,\cdots,i_k)\) is exactly 

\((\delta_{\mathrm{tri}} g)((i_1,\cdots,i_k); b_1, \dots, b_{n+1})\).Therefore,
\[
\delta_{\mathrm{HH}}(\Psi g) =\sum_{k=1}^{n} \sum_{1 \le i_1 < \dots < i_k \le n} P(a_1, \dots, a_{n}; i_1,\cdots,i_k) \otimes (\delta_{\mathrm{tri}} g)((i_1,\cdots,i_k); \dots) = \Psi(\delta_{\mathrm{tri}} g).
\]

This proves that \(\Psi\) is a cochain map.

   \textbf{ Injectivity of $\Psi$}
Suppose $\Psi(g) = 0$, since $A$ is a free algebra, there $g=0.$

     \end{proof}
\begin{corollary}
    Given the canonical embedding of cochain complexes
\[
C^*_{\mathrm{tri}}(B,N) \hookrightarrow C^*_{\mathrm{Hoch}}(A\otimes B,A\otimes N),
\]
we obtain a short exact sequence of complexes:
\[
0 \to C^*_{\mathrm{tri}}(B,N) \hookrightarrow C^*_{\mathrm{Hoch}}(A\otimes B,A\otimes N)\to Q^* \to 0,
\]
where $ Q^* = C^*_{\mathrm{Hoch}}(A\otimes B,A\otimes N)/C^*_{\mathrm{tri}}(B,N)$ is the quotient complex.

Applying the cohomology functor yields the long exact sequence:
\[
\begin{aligned}
\cdots \to H^{n-1}_{\mathrm{tri}}(B,N) &\to HH^{n-1}(A\otimes B,A\otimes N) \to H^{n-1}(Q^*) \\
&\to H^n_{\mathrm{tri}}(B,N) \to HH^n(A\otimes B,A\otimes N) \to H^n(Q^*) \\
&\to H^{n+1}_{tri}(B,N) \to HH^{n+1}(A\otimes B,A\otimes N) \to \cdots
\end{aligned}
\]
\end{corollary}
\begin{example}\label{exm:5.7}
Let:
\begin{itemize}
\item $A = F\langle x_1,x_2,x_3\rangle$ be the free tri-algebra .
\item $B$ is the tri-dendriform algebra defined in \ref{exp:tridend}.
\end{itemize}
\end{example}
A Hochschild 2-cochain \( \Psi: (A \otimes B) \otimes (A \otimes B) \to (A \otimes B) \) can be expressed in terms of \( g \) as:
\[
\begin{aligned}
(\Psi g)(x_1 \otimes e, x_2 \otimes e) &=  x_1\ast x_2 \otimes g(1, e, e) + x_2\ast x_1 \otimes g(2, e, e)+x_1\bullet x_2 \otimes g((1,2),e,e)\\
&=\alpha (x_1\ast x_2 )\otimes e + \beta (x_2\ast x_1) \otimes e+ \gamma (x_1\bullet x_2) \otimes e
\end{aligned}
\]

We want \( \alpha, \beta, \gamma \) such that \( \delta_{HH} (\Psi g) = 0 \) (2-cocycle condition) when evaluated on \( x_1 \otimes e, x_2 \otimes e, x_3 \otimes e \).

For a 2-cochain \( \Psi g \), the Hochschild differential is:

\[
(\delta_{HH} (\Psi g))(u,v,w) = u \cdot (\Psi g)(v,w) - (\Psi g)(u v, w) + (\Psi g)(u, v w) - (\Psi g)(u,v) \cdot w.
\]

Here \( u = x_1 \otimes e \), \( v = x_2 \otimes e \), \( w = x_3 \otimes e \).

We compute each term.

1. First term: \( u \cdot \Psi(v,w) \):

Let \( v = x_2 \otimes e \), \( w = x_3 \otimes e \).

First compute \( \Psi(v,w) \):

\[
\Psi(x_2 \otimes e, x_3 \otimes e) 
= \alpha \, (x_2 \ast x_3) \otimes e \;+\; \beta \, (x_3 \ast x_2) \otimes e \;+\; \gamma \, (x_2 \bullet x_3) \otimes e.
\]

Now multiply \( u = x_1 \otimes e \) with \( \Psi(v,w) \):

Use product rule in \( A \otimes B \):

\[
(x_1 \otimes e) \cdot [ (x_2 \ast x_3) \otimes e ]
\]
First term of product: \( (a_1 \ast a_2) \otimes (b_1 \prec b_2) \):

Here \( a_1 = x_1, b_1 = e, a_2 = x_2 \ast x_3, b_2 = e \):

\[
= (x_1 \ast (x_2 \ast x_3)) \otimes (e \prec e) = (x_1 \ast (x_2 \ast x_3)) \otimes e.
\]

Second term of product: \( (a_2 \ast a_1) \otimes (b_1 \succ b_2) \):

Here \( a_2 = x_2 \ast x_3, a_1 = x_1, b_1 = e, b_2 = e \):

\[
= ((x_2 \ast x_3) \ast x_1) \otimes (e \succ e) = ((x_2 \ast x_3) \ast x_1) \otimes e.
\]

Third term of product: \( (a_1 \bullet a_2) \otimes (b_1 \cdot b_2) \):

Here \( a_1 = x_1, a_2 = x_2 \ast x_3, b_1 = e, b_2 = e \):

\[
= (x_1 \bullet (x_2 \ast x_3)) \otimes (e \cdot e) = (x_1 \bullet (x_2 \ast x_3)) \otimes (-e).
\]

So:

\[
(x_1 \otimes e) \cdot [ (x_2 \ast x_3) \otimes e ] 
= (x_1 \ast (x_2 \ast x_3)) \otimes e \;+\; ((x_2 \ast x_3) \ast x_1) \otimes e \;-\; (x_1 \bullet (x_2 \ast x_3)) \otimes e.
\]

Similarly for other parts of \( \Psi(v,w) \), so \( u \cdot \Psi(v,w) \) =  

\[
\alpha \big[ x_1 \ast (x_2 \ast x_3) + (x_2 \ast x_3) \ast x_1 - x_1 \bullet (x_2 \ast x_3) \big] \otimes e
\]
\[
+ \beta \big[ x_1 \ast (x_3 \ast x_2) + (x_3 \ast x_2) \ast x_1 - x_1 \bullet (x_3 \ast x_2) \big] \otimes e
\]
\[
+ \gamma \big[ x_1 \ast (x_2 \bullet x_3) + (x_2 \bullet x_3) \ast x_1 - x_1 \bullet (x_2 \bullet x_3) \big] \otimes e.
\]

2. Second term: \( -\Psi(uv, w) \)

First compute \( uv = (x_1 \otimes e)(x_2 \otimes e) \):

\[
uv = (x_1 \ast x_2 + x_2 \ast x_1 - x_1 \bullet x_2) \otimes e.
\]

Now \( -\Psi(uv, w) =  \)

\[
= -\alpha \, (x_1 \ast x_2) \ast x_3 \otimes e \;-\; \beta \, x_3 \ast (x_1 \ast x_2) \otimes e \;-\; \gamma \, (x_1 \ast x_2) \bullet x_3 \otimes e
\]
\[
- \alpha \, (x_2 \ast x_1) \ast x_3 \otimes e \;-\; \beta \, x_3 \ast (x_2 \ast x_1) \otimes e \;-\; \gamma \, (x_2 \ast x_1) \bullet x_3 \otimes e
\]
\[
+ \alpha \, (x_1 \bullet x_2) \ast x_3 \otimes e \;+\; \beta \, x_3 \ast (x_1 \bullet x_2) \otimes e \;+\; \gamma \, (x_1 \bullet x_2) \bullet x_3 \otimes e.
\]

3. Third term: \( +\Psi(u, vw)= \)

\[
= \alpha \, x_1 \ast (x_2 \ast x_3) \otimes e \;+\; \beta \, (x_2 \ast x_3) \ast x_1 \otimes e \;+\; \gamma \, x_1 \bullet (x_2 \ast x_3) \otimes e
\]
\[
+ \alpha \, x_1 \ast (x_3 \ast x_2) \otimes e \;+\; \beta \, (x_3 \ast x_2) \ast x_1 \otimes e \;+\; \gamma \, x_1 \bullet (x_3 \ast x_2) \otimes e
\]
\[
- \alpha \, x_1 \ast (x_2 \bullet x_3) \otimes e \;-\; \beta \, (x_2 \bullet x_3) \ast x_1 \otimes e \;-\; \gamma \, x_1 \bullet (x_2 \bullet x_3) \otimes e.
\]

4. Fourth term  \( -\Psi(u,v) \cdot w \) =

\[
-\alpha \big[ (x_1 \ast x_2) \ast x_3 + x_3 \ast (x_1 \ast x_2) - (x_1 \ast x_2) \bullet x_3 \big] \otimes e
\]
\[
- \beta \big[ (x_2 \ast x_1) \ast x_3 + x_3 \ast (x_2 \ast x_1) - (x_2 \ast x_1) \bullet x_3 \big] \otimes e
\]
\[
- \gamma \big[ (x_1 \bullet x_2) \ast x_3 + x_3 \ast (x_1 \bullet x_2) - (x_1 \bullet x_2) \bullet x_3 \big] \otimes e.
\]

 Combine all terms and group by \( \alpha, \beta, \gamma \) and by the type of operation \( \ast \) or \( \bullet \).

We now have:
\[
\begin{aligned}
 \delta (\Psi g)(x_1\otimes e, x_2\otimes e, x_3\otimes e) &=(\beta+\gamma)(x_1\ast x_2 \ast x_3)e\\
&-(\alpha+\gamma)(x_3\ast x_2 \ast x_1)e\\
&+(\alpha-\beta)((x_1\bullet x_3) \ast x_2)e.
\end{aligned}
\]
We can group by \( \alpha, \beta, \gamma \) and by the type of operation \( \ast \) or \( \bullet \).

The condition $\delta_{HH}(\Psi g)=0$ leads to :

\[
\alpha = \beta=-\gamma.
\]

Thus, cocycle defined as:

\[
\begin{aligned}
(\Psi g)(x_1 \otimes e, x_2 \otimes e)
&=\alpha \Big(x_1\ast x_2  + x_2\ast x_1- x_1\bullet x_2\Big) \otimes e 
\end{aligned}
\]
To compute the Hochschild coboundary \(\delta_{HH} (\Psi f)\) evaluated at \((x_1 \otimes e, x_2 \otimes e)\), we use the Hochschild differential formula for a 1-cochain \(\Psi f\):

\[
(\delta_{HH} (\Psi f))(u, v) = u \cdot (\Psi f)(v) - (\Psi f)(u \cdot v) + (\Psi f)(u) \cdot v.
\]

1. First term: \( (x_1 \otimes e) \cdot (\Psi f)(x_2 \otimes e) \):
     \[
     (\Psi f)(x_2 \otimes e) =\lambda x_2 \otimes  e.
     \]
   - Now multiply by \((x_1 \otimes e_1)\) using the product in \(A \otimes B\):
     \[
     (x_1 \otimes e) \cdot \left(\lambda x_2 \otimes  e \right) =\lambda \Big ( x_1\ast x_2 + x_2\ast x_1 -x_1\bullet x_2 \Big) \otimes e.
     \]

2. Second term: \( -(\Psi f)((x_1 \otimes e) \cdot (x_2 \otimes e)) \):
     \[
     -\lambda \Big ( x_1\ast x_2 + x_2\ast x_1 -x_1\bullet x_2 \Big) \otimes e.
     \]
  
3. Third term: \( (\Psi f)(x_1 \otimes e) \cdot (x_2 \otimes e) \):
     \[
     +\lambda \Big ( x_1\ast x_2 + x_2\ast x_1 -x_1\bullet x_2 \Big) \otimes e.
     \]

{Combining all terms}
\[
\begin{aligned}
(\delta_{HH} (\Psi f))(x_1 \otimes e, x_2 \otimes e) = \lambda \Big ( x_1\ast x_2 + x_2\ast x_1 -x_1\bullet x_2 \Big) \otimes e.
\end{aligned}
\]
    Comparing \(g\) with $\delta_{\text{tri}} f$, we see that:

For $f$ to be a coboundary, we must have 
\[
(\Psi g)(x_1\otimes e,x_2\otimes e) = (\delta_{\text{HH}} (\Psi f))(x_1\otimes e,x_2\otimes e),
\]
which implies 
\[
\alpha=\lambda
\]

Thus, any such $\Psi g$ can be expressed as $\delta_{\text{HH}} (\Psi f)$ for some $f$, meaning $\Psi g$ is a coboundary.
So:
\[
H^2(B,B)=0.
\] 
\subsection*{Acknowledgments}
The author is grateful 
to  Kolesnikov P.S.
 and for discussions and useful comments.

\subsection*{Data Availability Statement}
Data sharing is not applicable to this 
article as no new data were created or analyzed in this study.

\end{document}